\documentclass[12pt]{amsart}
\usepackage{amscd}
\def\NZQ{\Bbb}               
\def\NN{{\NZQ N}}

\def\A{{\mathcal A}}

\def\frk{\frak}               

\def\Phi{{\frk n}}
\def\Phi{{\frk N}}
\def\opn#1#2{\def#1{\operatorname{#2}}} 
\opn\chara{char} \opn\length{\ell} \opn\pd{pd} \opn\rk{rk}
\opn\projdim{proj\,dim} \opn\injdim{inj\,dim} \opn\rank{rank}
\opn\depth{depth} \opn\grade{grade} \opn\height{height}
\opn\embdim{emb\,dim} \opn\codim{codim}

\opn\Tr{Tr} \opn\bigrank{big\,rank}
\opn\superheight{superheight}\opn\lcm{lcm}
\opn\trdeg{tr\,deg}
\opn\reg{reg} \opn\lreg{lreg} \opn\ini{in} \opn\lpd{lpd}
\opn\size{size}
\opn\div{div} \opn\Div{Div} \opn\cl{cl} \opn\Cl{Cl}
\opn\Spec{Spec} \opn\Supp{Supp} \opn\supp{supp} \opn\Sing{Sing}
\opn\Ass{Ass} \opn\Min{Min}
\opn\Ann{Ann} \opn\Rad{Rad} \opn\Soc{Soc}
\opn\Im{Im} \opn\Ker{Ker} \opn\Coker{Coker} \opn\Am{Am}
\opn\Hom{Hom} \opn\Tor{Tor} \opn\Ext{Ext} \opn\End{End}
\opn\Aut{Aut} \opn\id{id}

\opn\nat{nat}
\opn\pff{pf}
\opn\Pf{Pf} \opn\GL{GL} \opn\SL{SL} \opn\mod{mod} \opn\ord{ord}
\opn\Gin{Gin} \opn\Hilb{Hilb}
\opn\aff{aff} \opn\con{conv} \opn\relint{relint} \opn\st{st}
\opn\lk{lk} \opn\cn{cn} \opn\core{core} \opn\vol{vol}
\opn\link{link} \opn\star{star}
\opn\gr{gr}

\def\pot#1#2{#1[\kern-0.28ex[#2]\kern-0.28ex]}

\opn\dirlim{\underrightarrow{\lim}}
\opn\inivlim{\underleftarrow{\lim}}
\let\sect=\cap

\let\Union=\bigcup

\let\to=\rightarrow

\def\Implies{\ifmmode\Longrightarrow \else
        \unskip${}\Longrightarrow{}$\ignorespaces\fi}
\def\implies{\ifmmode\Rightarrow \else
        \unskip${}\Rightarrow{}$\ignorespaces\fi}
\def\iff{\ifmmode\Longleftrightarrow \else
        \unskip${}\Longleftrightarrow{}$\ignorespaces\fi}

\let\:=\colon
\newtheorem{Theorem}{Theorem}[section]
\newtheorem{Lemma}[Theorem]{Lemma}
\newtheorem{Corollary}[Theorem]{Corollary}
\newtheorem{Proposition}[Theorem]{Proposition}
\newtheorem{Remark}[Theorem]{Remark}

\newtheorem{Example}[Theorem]{Example}

\newtheorem{Definition}[Theorem]{Definition}

\opn\Syz{Syz} \opn\Im{Im} \opn\Ker{Ker} \opn\Coker{Coker}
\opn\Am{Am} \opn\Hom{Hom} \opn\Tor{Tor} \opn\Ext{Ext} \opn\End{End}
\opn\Aut{Aut} \opn\id{id}

\opn\nat{nat}
\opn\pff{pf}
\opn\Pf{Pf} \opn\GL{GL} \opn\SL{SL} \opn\mod{mod} \opn\ord{ord}
\opn\Gin{Gin}\opn\min{min}
\opn\Hilb{Hilb}\opn\adeg{adeg}\opn\std{std}\opn\ip{infpt}
\opn\Pol{Pol}\opn\sdepth{sdepth}\opn\infpt{infpt}
\opn\depth{depth}\opn\sqdepth{sqdepth}\opn{\Mon}{Mon}

\let\epsilon\varepsilon
\let\phi=\varphi
\let\kappa=\varkappa
\def\qed{\ifhmode\textqed\fi
      \ifmmode\ifinner\quad\qedsymbol\else\dispqed\fi\fi}
\def\textqed{\unskip\nobreak\penalty50
       \hskip2em\hbox{}\nobreak\hfil\qedsymbol
       \parfillskip=0pt \finalhyphendemerits=0}
\def\dispqed{\rlap{\qquad\qedsymbol}}

\opn\dis{dis}
\def\pnt{{\raise0.5mm\hbox{\large\bf.}}}

\opn\Lex{Lex}

\begin{document}

\title{Stanley Decompositions and Polarization}

\author{ Sarfraz Ahmad }

\address{Sarfraz Ahmad, COMSATS Institute of Information Technology,
    Lahore,Pakistan.}\email{sarfraz11@gmail.com}
  \maketitle

\begin{abstract}
We define nice partitions of the multicomplex associated to a Stanley ideal.
As the main result we show that if the monomial ideal $I$ is a CM Stanley ideal,  then  $I^p$ is a Stanley ideal as well, where $I^p$ is the polarization of $I$. \vskip 0.4 true cm

 \noindent
{\it Key words} : Monomial Ideals, Partitionable Simplicial Complexes, Multicomplexes, Stanley
Ideals, Polarization.\\
{\it 2010 Mathematics Subject Classification: Primary 13H10, Secondary 13C14, 13F20, 13F55.}\\
\end{abstract}

\section*{Introduction}
Let $K$ be a field and $S=K[x_1,\ldots,x_n]$ be a polynomial ring in $n$ variables. Let $I\subset S$ be a monomial ideal, $u\in S/I$ be a monomial and $Z\subseteq \{x_1,\ldots,x_n\}$. We denote by $uK[Z]$ the $K$-subspace of $S/I$ generated by all elements $uv$ where $v$ is a monomial in $K[Z]$. The $K$-subspace $uK[Z]\subset S/I$ is called a {\em Stanley space} of dimension $|Z|$, if $uK[Z]$ is a free $K[Z]$-module. A decomposition of $S/I$ as a finite direct sum of Stanley spaces $\mathcal{P}:S/I=\oplus_{i=1}^r u_iK[Z_i]$ is called a {\em Stanley decomposition}.
Stanley \cite{St} conjectured that there always exists such a decomposition  such that $|Z_i|\geq \depth(S/I)$. If Stanley conjecture holds for $S/I$ then $I$ is
called a {\em Stanley ideal}. The conjecture is still open but true in some special cases
\cite{An1}, \cite{An2}, \cite{Ci1}, \cite{Ci2}, \cite{HJY}, \cite{HVZ}, \cite{Na}, \cite{Po1}, \cite{Po2}, \cite{As}.

Let $\Gamma$ be a subset of $\mathbb{N}^n_{\infty}$. An element $m\in \Gamma$ is called maximal if there is no $a\in \Gamma$ with $a>m$. We denote by $\mathcal{M}(\Gamma)$ the set of maximal elements of $\Gamma$. If $a\in \Gamma$, we call $\infpt(a)=\{i:a(i)=\infty\}$. An element $a\in\Gamma$ is called a facet of $\Gamma$ if for  all $m\in \mathcal{M}(\Gamma)$ with $a\leq m$ one has  $|\infpt(a)|=|\infpt(m)|$. Herzog and Popescu {\cite{HP}} modify the Stanley's definition of multicomplexes {\cite{St}}. $\Gamma$ is called a multicomplex if for all $a\in \Gamma$ and for all $b\in \mathbb{N}^{n}_{\infty}$ with $b\leq a$ it follows that $b\in \Gamma$ and for all $a\in\Gamma$ there is a maximal element $m$ in $\Gamma$ such that $a\leq m.$ We define an interval $\mathcal{I}$ of $\Gamma$ as a subset of $\Gamma$ for which there exists $a\leq b$ in $\Gamma$ such that
$\mathcal{I}=[a,b]=\{c\in \Gamma: a\leq c\leq b\}$. A partition $\mathcal{P}: \Gamma=\bigcup_{i=1}^{t}[a_i,b_i]$ of $\Gamma$ is a presentation of $\Gamma$ as a finite disjoint union of intervals $[a_i,b_i]$.

Monomial ideals $I$ in the polynomial ring $S=K[x_1,\ldots,x_n]$ and multicomplexes in $\mathbb{N}_{\infty}^{n}$ correspond each other bijectively. The multicomplex associated to a monomial ideal $I$ is denoted by $\Gamma(I)$ and similarly, $I(\Gamma)$ denotes the monomial ideal associated to the multicomplex $\Gamma$. We show that Stanley's conjecture holds for $S/I$ if and only if there exists a partition  of the multicomplex $\Gamma(I)$ such that $|\infpt(b_i)| \geq \depth(S/I)\text{ for all }\ i.$ Any partition of a multicomplex satisfying this condition will be called nice.

Let $I\subset S=K[x_1,\ldots,x_n]$ be a monomial ideal and $\Gamma(I)$ be the multicomplex associated to $I$. In (Proposition 1.3), we show that a partition
$\mathcal{P}: \Gamma(I)=\cup_{i=1}^t[a_i,b_i]$ of $\Gamma(I)$ is nice if all $b_i$'s are
facets of $\Gamma(I)$. Also, when $S/I$ is Cohen-Macaulay, we have this result in
both directions (see Corollary 1.4).

Let $I^p$ be the polarization of the monomial ideal $I$ and let
$\Gamma^p$ be the multicomplex associated to $I^p$. In Theorem 2.4, we prove that in case of Cohen-Macaulay monomial ideals, if $\Gamma$ has a nice partition then $\Gamma^p$ has a nice partition. The converse of this theorem is still open.

\medskip
\noindent
\textbf{Acknowledgements.} The author like to thank Professor J\"{u}rgen Herzog for his valuable suggestions which improves the final form of the paper.

\section{Partitions of Multicomplexes}

  Let $\Gamma$ be a subset of $\mathbb{N}^n$. We define on
 $\mathbb{N}^n$ the partial order given by $$(a(1),\ldots,a(n))\leq (b(1),\ldots,b(n))$$ if $a(i)\leq b(i)$ for all $i$. According to
 Stanley \cite{St} $\Gamma$ is a multicomplex if for all $a\in \Gamma$ and all
 $b\in \mathbb{N}^n$ with $b \leq a$, it follows that $b\in \Gamma$.
 The elements of $\Gamma$ are called faces.

  Herzog and Popescu \cite{HP} modify the Stanley's definition of
 multicomplexes. Before giving this definition we
 introduce some notations. We set $\mathbb{N}_{\infty}=\mathbb{N}\cup
 \{\infty\}$. As usual we set $a \leq \infty$ for all $a\in
 \mathbb{N}$, and extend the partial order on $\mathbb{N}^n$
 naturally to $\mathbb{N}^n_{\infty}$. Thus now we take $\Gamma$ as a subset of $\mathbb{N}^n_{\infty}.$

  An element $m\in \Gamma$ is called {\em maximal} if there is no $a\in
  \Gamma$ with $a>m$. We denote by $\mathcal{M}(\Gamma)$ the set of
  maximal elements of $\Gamma$. If $a\in \Gamma$, we call
  $$\infpt(a)=\{i:a(i)=\infty\}$$ the {\em infinite part} of $a$.
\begin{Definition}
A subset $\Gamma \subset \mathbb{N}^n_{\infty}$ is called a
multicomplex if
\begin{enumerate}
\item for all $a\in \Gamma$ and for all $b\in \mathbb{N}^n_\infty$
 with $b\leq a$ it follows that $b\in \Gamma$,

\item for all $a\in \Gamma$ there exists an element $m\in
 \mathcal{M}(\Gamma)$ such that $a\leq m$.
\end{enumerate}
 \end{Definition}
 An element $a\in \Gamma$ is called a {\it facet} of $\Gamma$ if for
 all $m\in \mathcal{M}(\Gamma)$ with $a\leq m$ one has
 $\infpt(a)=\infpt(m)$. The set of all facets of $\Gamma$ will be
 denoted by $\mathcal{F}(\Gamma)$. In \cite{HP} it is shown that
 each multicomplex has only a finite number of facets.

 Monomial ideals $I$  in the polynomial ring  $ S=K[x_1,\ldots, x_n]$ and multicomplexes in $\NN_\infty^n$ correspond each other bijectively. The bijection is defined as follows:  Let $\Gamma$ be a multicomplex, and  let $I(\Gamma)$ be the $K$-subspace in $S=K[x_1,\ldots,x_n]$ spanned by all monomials $x^a$ such that $a\not\in \Gamma$.
Note that if $a\in \NN_\infty^n$ and $b\in \NN_\infty^n\setminus\Gamma$, then $a+b\in \NN_\infty^n\setminus \Gamma$, that is, if $x^b\in I(\Gamma)$ then $x^ax^b\in I(\Gamma)$ for all $x^a\in S$. In other words, $I(\Gamma)$ is a monomial ideal. In particular, the monomials $x^a$ with $a\in \Gamma$ form a $K$-basis of $S/I(\Gamma)$.

Conversely, given an arbitrary monomial ideal $I\subset S$, there is a unique multicomplex $\Gamma$ with $I=I(\Gamma)$, namely the smallest multicomplex (with respect to inclusion) which contains $A =\{a\in \NN_\infty^n\:
x^a\not\in I\}$. Such a multicomplex exists and is uniquely determined since an arbitrary intersection of multicomplexes is again a multicomplex.

One has the following obvious rules: let $\{\Gamma_j, j\in J\}$ be a family of multicomplexes. Then
 \begin{enumerate}
\item[(a)]$I(\bigcap\limits_{j\in J}\Gamma_j)=\sum\limits_{j\in J}I(\Gamma_j)$,

\item[(b)] if $J$ is finite, then $I(\bigcup\limits_{j\in J}\Gamma_j)=\bigcap\limits_{j\in J}I(\Gamma_j)$
\end{enumerate}

Let $\Gamma \subset \mathbb{N}^n_\infty$ be a multicomplex. We
define an {\it interval} $\mathcal{I}$ of $\Gamma$   as a subset
of $\Gamma$ for which there exists $a\leq b$ in $\Gamma$ such that
$\mathcal{I}=\{c\in \Gamma: a\leq c\leq b\}$. We denote an interval given by faces $a$ and $b$
as $[a,b]$.
A {\em partition} $\mathcal{P}$ of $\Gamma$ is a
presentation of $\Gamma$ as a finite disjoint union of intervals.

\begin{Lemma}
\label{finite}
Let  $\mathcal{P}: \Gamma=\Union_{i=1}^t[a_i,b_i]$ be a partition of $\Gamma$. Then $\infpt(a_i)=\emptyset$ for all $i$.
\end{Lemma}

\begin{proof} Assume that for some $i$,  say for $i=1$, we have $\infpt(a_1)\neq \emptyset$. We may assume that $a_1(1)=\infty$. Set $a=a_1$ and let $c$ be any integer. None of the faces $(c,a(2),\ldots,a(n))$ belong to $[a_1,b_1]$. Thus for each $c$ there exist an $i\in \{2,\ldots,t\}$ such that $(c,a(2),\ldots, a(n))\in [a_i,b_i]$. Hence for some $j>1$, infinitely of the vectors $(c,a(2),\ldots, a(n))$ belong to $[a_j,b_j]$. This is only possible if   $(\infty,  a(2),\ldots, a(n))$ belongs to $[a_j,b_j]$. This is  a contradiction,  since $a_1=(\infty,  a(2),\ldots, a(n))\in [a_1,b_1]$.
\end{proof}

Let $I\subset S=K[x_1,\ldots,x_n]$ be a monomial ideal. Any
decomposition of $S/I$ as a direct sum of $K$-vector spaces of the
form $uK[Z]$, where $u$ is a monomial in $K[X]$ and $Z\subset
X=\{x_1,\ldots,x_n\}$, is called a {\it Stanley decomposition} if $uK[Z]$ is a free $K[Z]$-module. In this paper we will call $uK[Z]$ a Stanley space of dimension $|Z|$, where $|Z|$ denotes the cardinality of $Z$.  Stanley decomposition have been studied in various combinatorial and algebraic contexts, see \cite{An1}, \cite{An2}, \cite{Ci1}, \cite{Ci2}, \cite{HJY}, \cite{HVZ}, \cite{Na}, \cite{Po1}, \cite{Po2}, \cite{As}.

Stanley \cite{St} conjectured that there always exists a Stanley
decomposition $$S/I=\bigoplus\limits_{i=1}^{r}u_iK[Z_i],$$
 such that $|Z_i|\geq \depth(S/I)$ for all $i$.

Next we describe how Stanley decompositions and partitions are related to each other. Let $\Gamma\subset \NN_\infty^n$ be a multicomplex, $[a,b]\subset \Gamma$ an interval and $U_{[a,b]}$ the $K$-subspace  of $S$ generated by all monomials $u=x_1^{c(1)}\cdots x_n^{c(n)}$ such that $c=(c(1),\ldots,c(n))\in [a,b]$. Then obviously one has that $U_{[a,b]}$
 is a Stanley space if and only if
 \begin{enumerate}
\item[(i)] $\infpt(a)=\emptyset$,

\item[(ii)] $i\not\in \infpt(b)$ $\Rightarrow$ $a(i)=b(i)$.
 \end{enumerate}
Indeed in this case  $U_{[a,b]}=x^{a}[Z_b]$, where $Z_b=\{x_i\:\, b(i)=\infty\}$.\\
Let $I\subset S$ be a monomial ideal and $\Gamma(I)$ be the multicomplex associated to $I$. Also let $S/I=\oplus_{i=1}^{r}x^{a_i}K[Z_i]$ be a Stanley decomposition of $S/I$. Set $b_i(j)=\infty$ if $x_j\in Z_i$ and $b_i(j)=a_i(j)$ if $x_j\not\in Z_i$.
Then  $\cup_{i=1}^{r}[a_i,b_i]$ is a partition of $\Gamma(I)$. For instance, if $a\in [a_i,b_i]\cap[a_j,b_j]\cap \NN^n$ for $i,j\in \{1,\ldots,r\}$ and $i\not=j$, then  $x^a\in {a_i}K[Z_i]\cup {a_j}K[Z_j]$, a contradiction. Thus $\cup_{i=1}^{r}[a_i,b_i]$ is disjoint. \\
Conversely, we observe that each interval $[a,b]$ with $\infpt(a)=\emptyset$ can be written as disjoint union of intervals
\begin{eqnarray}
\label{written}
[a,b]=\cup[c_i,b_i]
\end{eqnarray}
such that each $[c_i,b_i]$ corresponds to a Stanley space. Indeed, if, as we may assume,  for some integer $r$ we have that  $b(k)<\infty $ for $k\leq r$ and $b(k)=\infty$ for $k>r$. Then $[a,b]$ is the disjoint union of the intervals $$[(c(1),\ldots, c(r),a(r+1),\ldots, a(n)),(c(1),\ldots, c(r),\infty,\ldots, \infty)]$$ with $a(k)\leq c(k)\leq b(k)$ for $k=1,\ldots,r$, and each of these intervals satisfies (i) and (ii).
Therefore, due to (\ref{written}) and  Lemma~\ref{finite},  Stanley's conjecture holds  for $S/I$ if and only if there exists a
partition $\mathcal{P}: \Gamma=\bigcup\limits_{i=1}^{t}[a_i,b_i]$ of the multicomplex
$\Gamma=\Gamma(I)$ such that
\begin{eqnarray}
\label{nice}
 |\infpt(b_i)| \geq \depth(S/I(\Gamma)) \quad \text{for all}\quad i.
\end{eqnarray}

Any partition of a multicomplex satisfying condition (\ref{nice}) will be called {\em nice}.

\begin{Proposition}
\label{refine}
A partition $\mathcal{P}:\Gamma  =\Union_{i=1}^t[a_i,b_i]$ of the
multicomplex $\Gamma$ is a nice partition if $b_i\in
\mathcal{F}(\Gamma)$ for all $i$.
\end{Proposition}
\begin{proof}
 Let $I(\Gamma)=\bigcap_{i=1}^{m}Q_i$ be the unique irredundant presentation of $I$ as an intersection irreducible monomial ideals, and let $P_i=\sqrt{Q_i}$ for $i=1,\ldots,m$.  Then  $\Ass(S/I)=\{P_1,\ldots,P_m\}$.

 By \cite[Proposition 9.12]{HP} there is a bijection between the $Q_i$ and the set of $\mathcal{M}(\Gamma)$ of maximal faces of $\Gamma$. In fact, for each $i$ there is a unique $m_i\in \mathcal{M}(\Gamma)$ such that $Q_i=I(\Gamma(m_i))$ where $\Gamma(m_i)$ denotes the smallest multicomplex containing $m_i$. The assignment $Q_i\mapsto m_i$ establishes this bijection. Moreover, $\dim S/P_i=\infpt(m_i)$ for all $i$.
Therefore,
\begin{eqnarray*}
\min\{|\infpt(b_i)|\: b_i\in \mathcal{F}(\Gamma)\}&=&
 \min\{|\infpt(m_j)|\: m_j\in \mathcal{M}(\Gamma)\}\\
 &=&\min\{\dim(S/P_j)\: P_j\in \Ass(S/I(\Gamma))\}\\
 &\geq &
 \depth(S/I(\Gamma)).
\end{eqnarray*}
The first equation follows from the definition of the facets, while the last inequality is a basic fact of commutative algebra, see \cite[Proposition 1.2.13]{BH2}. These considerations show that our given  partition is nice.
\end{proof}

\begin{Corollary}
 Let $I\subset S$ be a monomial ideal such that $S/I$ is Cohen-Macaulay.
 Let $\Gamma$ be the multicomplex associated to $I$ and
 $\mathcal{P}\: \Gamma =\bigcup\limits_{i=1}^{t}[a_i,b_i]$ be a partition of
$\Gamma$. Then the following conditions are equivalent.
\begin{enumerate}
\item[(a)] $\mathcal{P}$ is nice.

\item[(b)]$\{b_1,\ldots,b_t\}\subseteq \mathcal{F}(\Gamma)$

\item[(c)] $\mathcal{M}(\Gamma)\subseteq \{b_1,\ldots,b_t\}\subseteq \mathcal{F}(\Gamma)$
\end{enumerate}
\end{Corollary}
\begin{proof}(a)\implies (b): In case $S/I$ is Cohen-Macaulay we have $|\infpt(b)|\leq \depth(S/I)$ for all faces of $\Gamma$, and equality holds for $b$ if and only if $b$ is a facet. Thus $\mathcal P$ can be nice only if $\{b_1,\ldots,b_t\}\subseteq \mathcal{F}(\Gamma)$.

(b)\implies (c): Let  $m\in \mathcal{M}(\Gamma)$; then $m\in [a_i,b_i]$ for some $i$. Since $m\leq b_i$ and since $m$ is maximal it follows that  $m=b_i$.  Thus $\mathcal{M}(\Gamma)
\subseteq \{b_1,\ldots,b_t\}$.

(c)\implies (a) follows from Proposition~\ref{refine}.
\end{proof}
\begin{Remark}
{\em In the above Corollary if $\mathcal{P}$ is nice then we can refine
it in such a way that for the refinement
$$\mathcal{P}': \Gamma =\Union_{i=1}^{t'}[a'_i,b'_i]$$ we have
$\{ b'_1,\ldots, b'_{t'}\}=\mathcal{F}(\Gamma)$. To prove this fact we first observe that $|\infpt(a_i)|=0$ for all $i$, see Lemma 1.2. Since $\mathcal{F}(\Gamma)=\Union_{i=1}^t(\mathcal{F}(\Gamma)\sect [a_i,b_i])$, it is enough to write each interval $[a_i,b_i]$ as a disjoint union of intervals $\Union_{j=1}^{l}[c_j, e_j]$ where $\{e_1,e_2,\ldots,e_l\}=\mathcal{F}(\Gamma)\sect [a_i,b_i]$.\\
For simplicity, we may assume that $b_i(k)<\infty$ for $k\leq r$ and   $b_i(k)=\infty$  for $k>r$. Then $e\in [a_i,b_i]$ is a facet of $\Gamma$ if and only if $a_i(k)\leq e(k)\leq b_i(k)$ for $k\leq r$ and   $e(k)=\infty$  for $k>r$. Thus if we set $c_j(k)=e_j(k)$ for $k\leq r$ and $c_j(k)=a_i(k)$ for $k> r$, then $[a_i,b_i]= \Union_{j=1}^{l}[c_j, e_j]$ is the desired refinement of $[a_i,b_i]$.}
\end{Remark}
\goodbreak

\section{Partitions and Polarization}
    Let $S=K[x_1,\ldots,x_n]$ be the polynomial ring in $n$
    variables over the field $K$, and $u=\prod\limits_{i=1}^{n}x_i^{a_i}$
    be the monomial in $S$. Then
    $$u^p=\prod\limits_{i=1}^{n} \prod\limits_{j=1}^{a_i} x_{ij} \in
    K[x_{11},\ldots,x_{1a_1},\ldots,x_{n1},\ldots,x_{na_n}]$$
is called the {\it polarization} of $u$.

 Let $I$ be the monomial ideal in $S$ with monomial generators
 $u_1,\ldots,u_{r}$. Then $(u_1^p,\ldots,u_{r}^p)$ is called a {\it
 polarization} of $I$ and is denoted by $I^p$. It is known that $I$ is Cohen-Macaulay if and only if $I^p$ is Cohen-Macaulay. Indeed, the elements
$x_{ij}-x_{i1}$, $i=1,\ldots,n$ and $j=1,2, \ldots$ form a regular sequence on $T/I^p$, and $T/I^p$ modulo this regular sequence is isomorphic to $S/I$.

Let $I=(u_1,\ldots,u_{s})\subset S$ be a monomial ideal. We may assume that for each $i\in[n]$ there exists $j$ such that $x_i$ divides $u_j$. Let $u_j=x_1^{a_{j1}}\cdots x_n^{a_{jn}}$ for $j=1,\ldots, s$ and set $r_i=\max{a_{ji}\:\; j=1,\ldots, s}$ for $i=1,\ldots, n$.  Moreover we set
$r=\sum_{i=1}^{n} r_i.$

Let  $I=\bigcap_{i=1}^{t}Q_i$  be the unique irredundant presentation of $I$ as an intersection of irreducible monomial ideals. In particular, each $Q_i$ is
 generated by pure powers of some of the variables. Then $I^p=\bigcap_{i=1}^{t_1}
 Q_i^p$ is an ideal in the polynomial ring $$T=K[x_{11},\ldots,x_{1r_1},x_{21},\ldots,x_{n1},\ldots,x_{nr_n}]$$
 in $r$ variables.

 We denote by $\Gamma$, $\Gamma^p$, $\Gamma_i$ and $\Gamma_i^p$ the
 multicomplexes associated to $I$, $I^p$, $Q_i$ and $Q_i^p$,
 respectively, and by $\mathcal F$, ${\mathcal F}^p$, ${\mathcal F}_i$ and ${\mathcal F_i}^p$ the set of
 facets of $\Gamma$, $\Gamma^p$, $\Gamma_i$ and $\Gamma_i^p$,
 respectively.

 Each $\Gamma_i$ has only one maximal facet,
 say $m_i$, and $m_i(k)\leq r_k-1$ for all $k$ with  $m_i(k)\not=
 \infty.$ Moreover, ${\mathcal M}(\Gamma)=\{m_1,\ldots, m_t\}$. It follows that the set of facets of $\Gamma$ is a subset of the set
\[
\mathcal{B}=\{b\in \mathbb{N}_{\infty}^n:\text{$b(i)<r_i$ if $b(i)\neq \infty$}\}.
\]

We define the map
$$\beta:\mathcal{B}\rightarrow \{0,\infty \}^r,\quad  b\mapsto b',$$
where the components of the vectors $b'$ are indexed by pairs of numbers $ij$, where for each $i=1,\ldots,n$ the second index $j$ runs in the range $j=1,\ldots, r_i$.
The map $\beta$ is defined as follows:
\[
b'(ij)=\left\{ \begin{array}{ll}
0, &  \;\text{if $b(i)<\infty$ and $j=b(i)+1$},\\
\infty, &  \;\text{otherwise.}
 \end{array} \right.
\]
We quote the following result by Soleyman Jahan \cite[Proposition 3.8]{Ja}
\begin{Proposition}
With the above assumptions and notation the restriction of the map
$\beta$ to $\mathcal{F}$ induces a bijection  $\mathcal{F}\to \mathcal{F}^p$.
\end{Proposition}

The following example demonstrates this bijection: let
$I=(x_1^2,x_1x_2,x_3^2)=(x_1,x_3^2)\cap(x_1^2,x_2,x_3^2)\subset
K[x_1,x_2,x_3]$. Then the multicomplex $\Gamma$
associated to $I$ has facets
$$(0,\infty,0),(0,\infty,1),(1,0,0),(1,0,1),$$
while the multicomplex of the polarized ideal $$I^p=(x_{11}x_{12},x_{11}x_{21},x_{31}x_{32})\subset
K[x_{11},x_{12},x_{21},x_{31},x_{32}].$$ has the facets
$$(0,\infty,\infty,0,\infty),(0,\infty,\infty,\infty,0),
(\infty,0,0,0,\infty),(\infty,0,0,\infty,0).$$

\medskip
Let $\Gamma=\bigcup_{i=1}^{t}[a_i,b_i]$ be a nice
 partition of $\Gamma$ with
 $\mathcal{F}(\Gamma)=\{b_1,\ldots,b_t\}$. With the notation introduced above we have

 \begin{Lemma}
 \label{bound}
 $a_i(j)\leq r_j$ for all $i$ and $j$.
 \end{Lemma}

 \begin{proof}
 Suppose without loss of generality that $a_1(1)> r_1$. Then $b_1(1)=\infty$, because if  $b_1(1)<\infty$ it follows that $a_1(1)\leq b_1(1)<r_1$, a contradiction.  Now since $\Gamma =\Union_{i=1}^t[a_i,b_i]$ and since $a=(r_1,a_1(2),\ldots, a_1(n))\in \Gamma \setminus [a_1,b_1]$, there exists $i>1$ such that $a\in [a_i,b_i]$. As above $b_i(1)=\infty$ because if $b_i(a)<\infty$ then $r_1\leq b_i(1)<r_1$, which is not possible. Hence we conclude that $a_i\leq a<a_1<b_i$ $\Rightarrow$ $a_1\in [a_i,b_i]$, a contradiction.
 \end{proof}

We want to ``polarize" the nice  partition $\Gamma=\bigcup_{i=1}^{t}[a_i,b_i]$.
For this purpose we consider the set
 $\A=\{a\in \mathbb{N}:\text{$a(i)\leq r_i$}\}$ and the following map $\gamma:\A\rightarrow \{0,1\}^r$ with
 \[
\gamma(a)(ij)=\left\{ \begin{array}{ll}
0, &  \;\text{if  $j>a(i)$},\\
1, &  \;\text{otherwise.}
 \end{array} \right.
\]
We observe that $\gamma$ is injective. Indeed, for  $a\neq a'$ there exists $i$ such that $a(i)\neq a'(i)$, say, $a(i)< a'(i)$. Then $a(ij)=0$ for $j=a(i)+1$, while $a'(ij)=1$ for $j=a(i)+1$.

\medskip
Let $\mathcal{I}=[a,b]\subset\Gamma\subset\mathbb{N}^n_{\infty}$ be an interval
such that $a=(a(1),a(2),\ldots,a(n))$ and
$b=(b(1),b(2),\ldots,b(n))$. We define an {\it i-subinterval} as $$\{c
\in\mathbb{N}_{\infty}:a(i)\leq c\leq
b(i)\}$$ and denote it by $\mathcal{I}(i)=[a(i),b(i)]$.

\begin{Example}
Let $a,b\in \Gamma\subset\mathbb{N}^2_{\infty}$ $a=(2,5),b=(4,\infty)$, then \\
$\mathcal{I}(1)=[a(1),b(1)]=[2,4]$ i.e. $\mathcal{I}(1)=\{2,3,4\},$\\
$\mathcal{I}(2)=[a(2),b(2)]=[5,\infty]$ i.e. $\mathcal{I}(2)=\{5,6,\ldots\}.$\\
\end{Example}
Next we need the following elementary lemma.
\begin{Lemma}
Let $\mathcal{I}_1,\mathcal{I}_2$ be two intervals of a multicomplex
$\Gamma\subset\mathbb{N}_{\infty}^n$ such that $\mathcal{I}_1=[a,b]$ and
$\mathcal{I}_2=[c,d]$. Suppose $\mathcal{I}_1\cap \mathcal{I}_2=\emptyset$, then there exists $i$ such that $\mathcal{I}_1(i)\cap
\mathcal{I}_2(i)=\emptyset.$
\end{Lemma}
 Let $I\subset S$ be a monomial ideal and let
 $S/I=\bigoplus\limits_{i=1}^{r}u_iK[Z_i]$ be its Stanley
 decomposition, where $u_i=x^{a_i}$ for $i=1,\ldots,r$. Then the
 Hilbert series is given by $H(S/I)=\sum_{i=1}^r
t^{|a_i|}/(1-t)^{|Z_i|}$, where $|a_i|$ denotes the sum of the  components of $a_i$ and $|Z_i|$ the cardinality of $Z_i$. Thus if $\Gamma$ is the multicomplex
associated to $I$ and  $\Gamma=\bigcup\limits_{i=1}^{t}
[a_i,b_i]$ the corresponding  partition (with $b_i(j)=a_i(j)$ for $x_j\not\in Z_i$ and $b_i(j)=\infty$ for $x_j\in Z_i$), then $H(S/I)=\sum_{i=1}^r
t^{|a_i|}/(1-t)^{|b_i|_{\infty}},$ where
$|b_i|_{\infty}=|\infpt{b_i}|$.

\begin{Theorem} Let $I\subset S$ be a monomial ideal such
that $S/I$ is Cohen-Macaulay, and $I^p$ be the polarization of $I$.
 Suppose $I$ satisfies the Stanley Conjecture, then $I^p$ satisfies it too.
\end{Theorem}
\begin{proof} Let $\Gamma$ be the multicomplex
associated to $I$. Since $I$ satisfies the Stanley Conjecture, $\Gamma$ has a nice partition. Let $\Gamma^p$ be the multicomplex
associated to
$I^p$. Then we show that $\Gamma^p$ has a nice partition. \\Let
$\Gamma=\bigcup_{i=1}^{\hat{t}}[\hat{a}_i,\hat{b}_i]$ be a nice
partition of $\Gamma$ then by Corollary 1.4, $\hat{b}_i\in
\mathcal{F}(\Gamma)$ for all $i$.
 Again by Remark 1.5, we can refine this partition to another nice partition say
 $\mathcal{P}: \Gamma =\bigcup_{i=1}^{t}[a_i,b_i]$ such that
$\{ b_1,\ldots, b_{t}\}=\mathcal{F}(\Gamma)$.\\
Let $\beta$ and $\gamma$ be the functions  defined above and set
$\beta(b_i)=\bar{b}_i$ and $\gamma(a_i)=\bar{a}_i$ for all
$i=1,\ldots,t'$. We will show that
$\mathcal{P}^p:\Gamma^p=\bigcup_{i=1}^{t}[\bar{a}_i,\bar{b}_i]$
 is a nice partition of $\Gamma^p$.\\
$\mathcal{P}^p$ is a partition if the intervals
$[\bar{a}_i,\bar{b}_i]$ are disjoint for all $i=1,\ldots,t$ and
$\mathcal{P}^p$ covers all the faces of $\Gamma^p$.\\
Suppose that the intervals are not disjoint and say there exist a
face $a\in [\bar{a}_i,\bar{b}_i]\cap[\bar{a}_j,\bar{b}_j]$ for some
$i\not=j$, $i,j\in \{1,\ldots,t\}$. Since $a_i\not= a_j$ we get
$\bar{a}_i\not=\bar{a}_j$, $\gamma$
being injective. \\
The intervals $[a_i,b_i]$ and $[a_j,b_j]$ are disjoint and so by
Lemma 2.4, there exists at least one pair of $i_1-$subintervals say
$[a_i(i_1),b_i(i_1)]$ and $[a_j(i_1),b_j(i_1)]$ for
$i_1\in\{1,\ldots,n\}$ such that
$[a_i(i_1),b_i(i_1)]\cap[a_j(i_1),b_j(i_1)]=\emptyset.$\\
So at least one of $\, b_i(i_1),b_j(i_1) \, $ is finite say
$b_i(i_1)\not= \infty$, and so by condition $(ii)$ on page 4,
$b_i(i_1)=a_i(i_1)$. Also we can assume that
$a_i(i_1)< a_j(i_1).$ If not and $b_j(i_1)=\infty$
then $[a_i(i_1),b_i(i_1)]\subset [a_j(i_1),b_j(i_1)]$ which is not possible, if $b_j(i_1)< \infty$ then change $i$ by $j$.\\
Let $a_i(i_1)=b_i(i_1)=k$ and $a_j(i_1)=m>k.$ Then by definition of
$\gamma$ and $\beta$ we have
$\bar{a}_i(i_1\overline{k+1})=0=\bar{b}_i(i_1\overline{k+1})$ and
$\bar{a}_j(i_1l)=1$ for $l\leq m$, thus $\bar{a}_j(i_1\overline{k+1})=1.$\\
It follows $a(i_1\overline{k+1})=0.$ On the other hand since $a\geq
\bar{a}_j$ we get $a(i_1\overline{k+1})\geq\bar{a}_j(i_1\overline{k+1})=1.$ A
contradiction.\\
Now for the second part of the proof, we will use the Hilbert
series. We have  $H(S/I)=\sum_{i=1}^t
s^{|a_i|}/(1-s)^{|b_i|_{\infty}}$. The definition of the function
$\gamma$ implies  that $|a_i|=|\bar{a}_i|$ for all
$i=\{1,\ldots,t\}$. Now for each polarization step, the depth of
$S/I$ increases by 1. Also by definition of $\beta$ for each
polarization step, number of infinite points increases by 1. Thus
after $n_1$ polarization steps
$|infpt(\bar{b}_i)|=|infpt(b_i)|+n_1.$ So
$$H(\bigcup_{i=1}^t[\bar{a}_i,\bar{b}_i])=\sum\limits_{i=1}^t
\frac{s^{|a_i|}}{(1-s)^{|b_i|_{\infty}+n_1}}
=\frac{1}{(1-s)^{n_1}}H(S/I)
$$
is in fact the Hilbert series of $H(S^p/I^p)$. Hence
$S^p/I^p=\cup_{i=1}^{t}[\bar{a}_i,\bar{b}_i].$\\
Note that $\mathcal{P}^p$ is a nice partition because $|\bar{b}_i|_{\infty}=|b_i|_{\infty}+n_1\geq \depth_S(S/I)+n_1=\depth_{S^p}(S^p/I^p)$ for all $i$.
\end{proof}

The converse of above theorem is still open. If one can proof the converse, then the Stanley Conjecture will reduce to the case of squarefree monomial ideals $I$ where $I$ is Cohen Macaulay.

\end{document}